\documentclass[12pt]{amsart}
\usepackage{bbm}
\usepackage{mathrsfs}
\usepackage{stmaryrd}
\usepackage{cases}
\usepackage{amssymb}
\usepackage{amsmath}
\usepackage{amsfonts}
\usepackage{graphicx}
\usepackage{amsmath,amstext,amsbsy,amssymb}
\usepackage[table]{xcolor}
\usepackage{multirow}
\usepackage{cases}
\usepackage{color}

\newtheorem{theorem}{Theorem}[section]
\newtheorem{lemma}[theorem]{Lemma}

\newtheorem{proposition}{Proposition}[section]
\newtheorem{corollary}[theorem]{Corollary}

\theoremstyle{definition}
\newtheorem{definition}[theorem]{Definition}

\theoremstyle{remark}

\numberwithin{equation}{section} \errorcontextlines=0

\begin{document}

\title[Modules of the toroidal Lie algebra $\widehat{\widehat{\mathfrak{sl}}}_{2}$]
{Modules of the toroidal Lie algebra $\hat{\hat{\mathfrak{sl}}}_{2}$}
\author{Naihuan Jing}
\address{Department of Mathematics,
   North Carolina State University,
   Raleigh, NC 27695, USA}
\email{jing@math.ncsu.edu}

\author{Chunhua Wang*}
\address{School of Mathematics,
South China University of Technology, Guangzhou, Guangdong 510640, China}
\email{tiankon\_g1987@163.com}

\thanks{*Corresponding author.}

\begin{abstract}
Highest weight modules of the double affine Lie algebra $\widehat{\widehat{\mathfrak{sl}}}_{2}$ are studied under a
new triangular decomposition. Singular vectors of Verma modules are
determined using a similar condition with horizontal affine Lie
subalgebras, and highest weight modules are described under the
condition that $c_1>0$ and $c_2=0$.
\end{abstract}
\maketitle

\section{\bf Introduction}

Toroidal Lie algebras are multiloop generalization of the affine Lie
algebras. Their representations can be studied with similar but
distinctive methods as their affine counterparts.
Classification of irreducible integrable modules
with finite dimensional weight spaces has been carried out and properties of the integrable modules
have been investigated in
\cite{E1,E2,E3,E4,RJ}. Berman and Billig \cite{BB} constructed general modules by the standard
induction procedure and studied their irreducible quotients using
vertex operator techniques. As a special subalgebra of the toroidal Lie algebra, the
double affine algebras also have similar representation theory.
However, the integrable modules are no longer completely reducible
\cite{CT}.

In this paper, we use a different triangular decomposition to study
representations of the double affine Lie algebra $\mathfrak{T}$
\cite{JM}. We determine all singular vectors of the Verma
modules and give relatively easier description of their submodule
structures. When one canonical center is positive and the other
center is trivial, we are able to determine integrability and
irreducibility of the Verma modules, which then enable us to describe general highest
weight modules.

The paper is organized as follows. In section two, we describe a new
triangular decomposition of the double affine algebra of
$\mathfrak{sl_{2}}$. In section three, we study the Verma modules
based on the triangular decomposition. Using the affine Weyl
group, we fix the singular vectors of $M(\lambda)$ and study
integrability of the quotient $W(\lambda)$. In section four, we give
a necessary and sufficient condition for the irreducibility of
$M(\lambda)$ and $W(\lambda)$.

Throughout the paper, we will denote by $\mathbb{Z},\mathbb{Z}_{+}, \mathbb{N}$
the set of integers, nonnegative integers and positive integers,
respectively.

\section{\bf The toroidal Lie algebra
$\widehat{\widehat{\mathfrak{sl}}}_{2}$}

 Let
$\mathfrak{sl_2}(\mathbb{C})$ be the three dimensional simple Lie
algebra generated by $e, f, \alpha^{\vee}$ with the canonical
bilinear form given by $(e|f)=\frac
12(\alpha^{\vee}|\alpha^{\vee})=1$, where $\alpha$ is the simple
root. The toroidal Lie algebra
$\widehat{\widehat{\mathfrak{sl}}}_{2}=\mathfrak{sl_2}(\mathbb{C})\otimes
\mathbb C[t_1^{\pm 1}, t_2^{\pm
1}]\oplus\mathbb{C}c_1\oplus\mathbb{C}c_2\oplus\mathbb{C}d_1\oplus\mathbb{C}d_2$
is a central extension of the 2-loop algebra
$\mathfrak{sl_2}(\mathbb{C})\otimes \mathbb C[t_1^{\pm 1}, t_2^{\pm
1}]$ with the following Lie bracket:
\begin{subequations}
\begin{equation}
[x\otimes t^r,y\otimes t^s]=[x,y]\otimes
t^{r+s}+(x|y)\delta_{r_1,-s_1}\delta_{r_2,-s_2}(r_1 c_1+r_2 c_2),
\end{equation}
\begin{equation}
[x\otimes t^r,c_1]=[x\otimes t^r,c_2]=0,
\end{equation}
\begin{equation}
 [d_i,x\otimes t^r]=r_i
x\otimes t^r,\quad [d_i,c_j]=0
\end{equation}
\end{subequations}
where $x, y\in \mathfrak{sl_2}(\mathbb{C})$, $r=(r_1,r_2)\in
\mathbb{Z}^2$, $s=(s_1,s_2)\in \mathbb{Z}^2$, and
$t^r=t_1^{r_1}t_2^{r_2}$. In the following we also denote $x(m,n)=
x\otimes t_1^m t_2^n$.

The Cartan subalgebra is $\mathfrak{h}=\mathbb
C\alpha^{\vee}\oplus\mathbb{C}c_1\oplus\mathbb{C}c_2\oplus\mathbb{C}d_1\oplus\mathbb{C}d_2$,
where $c_1$ and $c_2$ are central elements. Let $\mathfrak{h}^{*}$
be the dual space of $\mathfrak{h}$. For a functional
$\beta\in\mathfrak{h}^{*}$, let $\mathfrak T_{\beta}=\{x\in\mathfrak
T|[h,x]=\beta(h)x, \, \forall \, h\in\mathfrak h\}$ be the root
subspace. The root system $\Delta$ of $\mathfrak T$ consists of all
nonzero $\beta\in\mathfrak{h}^{*}$ such that $\mathfrak
T_{\beta}\neq 0$.

Let $\delta_1,\delta_2,\omega_1,\omega_2\in \mathfrak{h}^{*}$ be the
linear functionals defined by $\delta_i (d_j)=\delta_{ij}$,
$\delta_i (c_j)=\delta_i (\alpha^{\vee})=0$ and $\omega_i
(c_j)=\delta_{ij}$, $\omega_i (d_j)=\omega_i (\alpha^{\vee})=0$ for
$i,j=1,2$. It is easy to check that the root system
$\Delta=\{\pm\alpha+\mathbb{Z}\delta_1+\mathbb{Z}\delta_2\}\cup(\{\mathbb{Z}\delta_1+\mathbb{Z}\delta_2\}\backslash
\{0\})$. Let $(\ \ |\ \ )$ be the invariant form on $\mathfrak{h}$
defined by
\[\begin{split}(\alpha^{\vee}|\alpha^{\vee})=2,~
(\alpha^{\vee}|c_i)=(\alpha^{\vee}|d_j)=0,\\
(c_i|d_j)=\delta_{ij},~(c_i|c_j)=(d_i|d_j)=0.\end{split}\]

Then the associated invariant form on $\mathfrak{h}^{*}$ is given by

\[\begin{split}(\alpha|\alpha)=&2,~(\alpha|\delta_i)=(\alpha|\omega_j)=0,\\
(\delta_i|\omega_j)=&\delta_{ij},~
(\delta_i|\delta_j)=(\omega_i|\omega_j)=0,\end{split}\] where
$i,j=1,2$. The real and imaginary roots are given respectively by
\begin{subequations}
\begin{equation}
\Delta^{re}=\{\pm\alpha+\mathbb{Z}\delta_1+\mathbb{Z}\delta_2\},
\end{equation}
\begin{equation}
\Delta^{im}=\{\mathbb{Z}\delta_1+\mathbb{Z}\delta_2\}\backslash\{0\}.
\end{equation}
\end{subequations}

Clearly
$\widehat{\mathfrak{g}_i}=\mathfrak{sl_2}(\mathbb{C})\otimes\mathbb{C}[t_i,t_i^{-1}]\oplus\mathbb{C}c_i\oplus\mathbb{C}d_i~(i=1,2)$
are two subalgebras of $\mathfrak T$ that are isomorphic to the
affine Lie algebra $A_1^{(1)}$. Denote the root system of
$\widehat{\mathfrak{g}_i}$ by $\Delta_{\widehat{\mathfrak{g}_i}}$.
It is well-known that $\Delta_{\widehat{\mathfrak{g}_1}}$ is
decomposed into positive and negative parts \cite{VG}:
$\Delta_{\widehat{\mathfrak{g}_1}}=\Delta_{\widehat{\mathfrak{g}_1}+}\cup\Delta_{\widehat{\mathfrak{g}_1}-}$,
where

\[\begin{aligned}\Delta_{\widehat{\mathfrak{g}_1}+}=&\{\alpha+\mathbb{Z}_{+}\delta_1\}\cup\{-\alpha+\mathbb{N}
\delta_1\}\cup\mathbb{N}\delta_1,\\
\Delta_{\widehat{\mathfrak{g}_1}-}=&\{-\alpha-\mathbb{Z}_{+}\delta_1\}\cup\{\alpha-\mathbb{N}\delta_1\}\cup\{-\mathbb{N}\delta_1\}.
\end{aligned}\]

Let $\alpha_1=\alpha$, $\alpha_0=\delta_1-\alpha$ and
$\alpha_{-1}=\delta_2-\alpha$, then all roots are integral linear
combination of $\alpha_i, i=-1, 0, 1$, and they are called the
``simple'' roots \cite{MS} of $\mathfrak T$. However some roots can
not be represented as negative or positive linear
combinations of the ``simple'' roots. In this paper, we view
$\mathfrak T$ as an affinization of the Lie algebra
$\widehat{\mathfrak{g}_1}$ and define the following partition of
$\Delta$:

\[\begin{aligned}\Delta_+=&\{\alpha+\mathbb{Z}_+ \delta_1+\mathbb{Z}_+
\delta_2\}\cup\{-\alpha+\mathbb{N}\delta_1+\mathbb{Z}_+
\delta_2\}\cup \{\mathbb{N}\delta_1+\mathbb{Z}_+ \delta_2\}\\
&\cup\{-\alpha-\mathbb{Z}_+
\delta_1+\mathbb{N}\delta_2\}\cup\{\alpha-\mathbb{N}\delta_1+\mathbb{N}\delta_2\}\cup
\{-\mathbb{N}\delta_1+\mathbb{N}\delta_2\}\cup\mathbb{N}\delta_2
\end{aligned}\]
and
\[\begin{aligned}
\Delta_{-}=&\{-\alpha-\mathbb{Z}_+ \delta_1-\mathbb{Z}_+
\delta_2\}\cup\{\alpha-\mathbb{N}\delta_1-\mathbb{Z}_+
\delta_2\}\cup \{-\mathbb{N}\delta_1-\mathbb{Z}_+ \delta_2\}\\
&\cup\{\alpha+\mathbb{Z}_{+}\delta_1-\mathbb{N}\delta_2\}\cup\{-\alpha+\mathbb{N}\delta_1-\mathbb{N}\delta_2\}\cup
\{\mathbb{N}\delta_1-\mathbb{N}\delta_2\}\cup\{-\mathbb{N}\delta_2\}.
\end{aligned}\]
The corresponding positive (resp. negative) root space is denoted by
$\mathfrak T_+=\bigoplus_{\beta\in\Delta_+}\mathfrak T_{\beta}$
(resp. $\mathfrak T_-=\bigoplus_{\beta\in\Delta_-}\mathfrak
T_{\beta}$). Then $\mathfrak T=\mathfrak T_+\oplus \mathfrak{h}
\oplus \mathfrak T_-$ is the associated triangular decomposition of
$\mathfrak T$.

Let $\mathcal {Q}_+=\mathbb Z_+$-span of $\Delta_+$. Similarly,
$\mathcal {Q}_{1+}=\mathbb Z_+$-span of
$\Delta_{\widehat{\mathfrak{g}_1}+}=\mathbb Z_+\alpha_0+\mathbb Z_+\alpha_1$ is the positive root lattice
of $\widehat{\mathfrak{g}_1}$. Let $\lambda,\mu\in\mathfrak{h}^{*}$.
We say that $\lambda\geq\mu$ if $\lambda-\mu$ is a nonnegative
linear combination of roots in $\Delta_+$.

The Weyl group of $\mathfrak T$ is defined as usual \cite{MS}.
\begin{definition}
For a real root $\beta=\pm\alpha+n_1\delta_1+n_2\delta_2$, we define
the reflection $r_{\beta}$ on $\mathfrak{h}^{*}$ by
\[r_{\beta}(\lambda)=\lambda-\lambda(\beta^{\vee})\beta,\] where
$\lambda\in\mathfrak{h}^{*}$ and
$\beta^{\vee}=\pm\alpha^{\vee}+n_1c_1+n_2c_2$. The Weyl group
$W_{\mathfrak T}$ is generated by $r_{\beta}~(\beta\in
\Delta^{re})$.
\end{definition}

\section{\bf The Verma module $M(\lambda)$}
In this section, we study highest weight modules of $\mathfrak
T$. Integrable modules are constructed by Chari \cite{CT} for double
affine Lie algebras, and a classification has been given by Rao
\cite{E1,E2,E3,E4} and Jiang \cite{RJ} for irreducible integrable
modules of the toroidal Lie algebras. We will take the new
triangular decomposition to study the Verma modules $M(\lambda)$.

\begin{definition} A module $M$ of $\mathfrak T$ is called a highest weight
module if there exists some $0\neq v\in M$ such that
\begin{enumerate}
\item the vector $v$ is a weight vector, that is $h.v=\lambda(h)v$ for some
$\lambda\in\mathfrak{h}^{*}$ and all $h\in\mathfrak{h}$,

\item $\mathfrak T_+.v=0$,

\item $U(\mathfrak T).v=M$.
\end{enumerate}
\end{definition}

\begin{definition}A module $M$ of $\mathfrak T$ is integrable if $M$
is a weight module and all $x_{\alpha}(m,n)$'s are locally
nilpotent, i.e. for any nonzero $v\in M$ there exists
$N=N(\alpha,m,n, v)$ such that $x_{\alpha}(m,n)^{N}.v=0$.
\end{definition}

\begin{definition}
A nonzero element $v\in M$ is called a singular vector if it is a
weight vector and $\mathfrak T_+.v=0$.
\end{definition}

Let $\lambda\in\mathfrak{h}^{*}$.  The one-dimensional vector space
$\mathbb{C}1_{\lambda}$ can be viewed as a $\mathfrak T_+\oplus
\mathfrak{h}-$module with $\mathfrak T_+.1_{\lambda}=0$ and
$h.1_{\lambda}=\lambda(h)\cdot1_{\lambda}$ for all $h\in
\mathfrak{h}$. Assume that the central element $c_1$ acts as a
scalar $k_1\geq 0$, and the other center $c_2$ acts trivially. Then we
have the induced Verma module:
\[M(\lambda)=U(\mathfrak
T)\otimes_{U(\mathfrak T_+\oplus \mathfrak{h})}
\mathbb{C}1_{\lambda},\]where $U(\mathfrak T)$ is the universal
enveloping algebra of $\mathfrak T$.

\begin{proposition}
\begin{enumerate}
\item $M(\lambda)$ is a $U(\mathfrak T_-)-$free module generated by the
highest weight vector: $1\otimes1_{\lambda}=v_{\lambda}$.

\item $\dim M(\lambda)_{\lambda}=1$; $0<\dim
M(\lambda)_{\lambda-\beta}<+\infty$ for every $\beta\in \mathcal
{Q}_{1+}$; Otherwise, $\dim M(\lambda)_{\lambda-\gamma}=\infty$ for
any $\gamma\in \mathcal {Q}_+$.

\item  The module $M(\lambda)$ has a unique irreducible quotient $L(\lambda)$.

\end{enumerate}
\end{proposition}

We now determine all singular vectors of $M(\lambda)$. In the
following, we will consider the properties of $M(\lambda)$ and
$L(\lambda)$ under the assumptions that the highest weight $\lambda$
is dominant on $\widehat{\mathfrak{g}_1}$.

\begin{proposition}\label{6}
If $\lambda(\alpha_i^{\vee})=n_i~(i=0,1)$ are nonnegative integers,
then $\sum\limits_{i=0}^1U(\mathfrak T_-).y_i^{n_i+1}v_{\lambda}$ is
a proper submodule of $M(\lambda)$, where $y_1=y, y_0=x\otimes
t_1^{-1}$.
\end{proposition}

\begin{proof} We need to show that
$y_i^{n_i+1}v_{\lambda}~(i=0,1)$ are singular vectors of
$M(\lambda)$, i.e, $\mathfrak T _+.y_i^{n_i+1}v_{\lambda}=0$ for
$i=0,1$. It suffice to show that $e(m,n)\in \mathfrak T_+$ and
$f(m^{'},n^{'})\in \mathfrak T_+$ act trivially on
$y_i^{n_i+1}.v_{\lambda}~(i=0,1)$. Because the weight of
$e(m,n)y_i^{n_i+1}.v_{\lambda}~(m\in\mathbb{Z}, n\in\mathbb{N})$ or
$f(m^{'},n^{'})y_i^{n_i+1}.v_{\lambda}~(m^{'}\in\mathbb{Z},
n^{'}\in\mathbb{N})$ is higher than $\lambda$, we get
$e(m,n)y_i^{n_i+1}.v_{\lambda}=f(m^{'},n^{'})y_i^{n_i+1}.v_{\lambda}=0~(m,m^{'}\in\mathbb{Z},
n,n^{'}\in\mathbb{N})$. Therefore, we only need to consider
$z.y_i^{n_i+1}.v_{\lambda}$ with $z\in\widehat{\mathfrak{g}_1}_+$.
But this is zero as it is the case of affine Lie algebras.
\end{proof}

Denote the quotient
$W(\lambda)=M(\lambda)/\sum\limits_{i=0}^1U(\mathfrak
T_-).y_i^{n_i+1}v_{\lambda}$.

Let $W_{\widehat{\mathfrak{g}_1}}$ be the Weyl group of the affine
Lie algebra $\widehat{\mathfrak{g}_1}$ defined as above. It has two
generators given by the simple reflections $r_{\alpha_0}$ and
$r_{\alpha_1}$. For arbitrary $w\in W_{\widehat{\mathfrak{g}_1}}$,
we define $w\cdot \lambda =w(\lambda+\rho)-\rho$, where $\rho$
satisfies $\rho(\alpha_i^{\vee})=1~(i=0,1)$.

\begin{corollary}For arbitrary $w\in W_{\widehat{\mathfrak{g}_1}}$, the module $M(w\cdot
\lambda)$ is a submodule of $M(\lambda)$.
\end{corollary}

\begin{proof} This can be proved by induction
on the length of $w$ as in Proposition \ref{6}.
\end{proof}

We now look at the integrability of $W(\lambda)$, where
$\lambda(\alpha_i^{\vee})=n_i\geq 0$, $i=0, 1.$

\begin{proposition}
Suppose that $c_1$ acts on $W(\lambda)$ as a scalar $k_1>0$ and
$c_2$ is trivial. Then $W(\lambda)$ is not integrable.
\end{proposition}

\begin{proof}
Let $W(\lambda)=U(\mathfrak T_-).w_{\lambda}$, where $w_{\lambda}$
is the image of $v_{\lambda}$ in $W(\lambda)$. We claim that
$e(0,-1)$ is not locally nilpotent. In fact, suppose that $e(0, -1)$
is nilpotent on $w_{\lambda}$ and assume that $N$ is the minimum
positive integer such that $e(0,-1)^N.w_{\lambda}=0$. Then we have
\begin{equation*}
\begin{aligned}
0=&f(0,1)e(0,-1)^N.w_{\lambda}\\=& [f(0,1),e(0,-1)^{N}].w_{\lambda}\\
=&-Ne(0,-1)^{N-1}((N-1)\cdot 1-(c_2-\alpha^{\vee})).w_{\lambda}\\
=&-N(N-1+n_1)e(0,-1)^{N-1}.w_{\lambda},
\end{aligned}
\end{equation*}
where we have used
$\lambda(\alpha_{-1}^{\vee})=\lambda(c_2-\alpha^{\vee})=-n_1$.

By the minimality of $N$ we obtain that $(N-1)+n_1= 0$, then $N=1$
and $n_1= 0$. Then $e(0,-1).w_{\lambda}=0$ and
\begin{equation}\label{7}
f(1,0)e(0,-1).w_{\lambda}=-\alpha^{\vee}(1,-1).w_{\lambda}=0.
\end{equation}

Applying $\alpha^{\vee}(-1,1)$ to Eq. (\ref{7}), we obtain
\begin{equation*}
\alpha^{\vee}(-1,1)\alpha^{\vee}(1,-1).w_{\lambda}=(\alpha^{\vee}|\alpha^{\vee})(-c_1+c_2).w_{\lambda}=-2k_1
w_{\lambda}=0.
\end{equation*}
This is a contradiction as we have assumed $k_1>0$. Therefore, the
quotient $W(\lambda)$ is not integrable.
\end{proof}

\begin{proposition}
Suppose that $c_1$ acts as a positive constant $k_1$ and $c_2$ is
trivial. Then some weight spaces of $W(\lambda)$ are infinite
dimensional.
\end{proposition}
\begin{proof} Observe that $\alpha^{\vee}(-m,-1).w_{\lambda}\neq 0$,
$m\in\mathbb{N}$. Otherwise, we have
$$\alpha^{\vee}(m,1)\alpha^{\vee}(-m,-1).w_{\lambda}=[\alpha^{\vee}(m,1),\alpha^{\vee}(-m,-1)].w_{\lambda}=2mk_{1}w_{\lambda}=0.$$
Similarly we also get that $\alpha^{\vee}(m,-1).w_{\lambda}\neq0$
for $m\in\mathbb{N}$.

We claim that
$\{\alpha^{\vee}(-m,-1)\alpha^{\vee}(m,-1).w_{\lambda},m\in\mathbb{N}\}$
is linearly independent. Suppose there exist $a_m\neq 0$ such that
\begin{equation}\label{*}
\sum\limits_m
a_m\alpha^{\vee}(-m,-1)\alpha^{\vee}(m,-1).w_{\lambda}=0.
\end{equation}
Let $s\in\{m|a_m\neq 0\}$. Applying $\alpha^{\vee}(s,1)$ to
Eq. (\ref{*}), we obtain

\[\begin{aligned}
0=&\sum\limits_m a_m([\alpha^{\vee}(s,1),\alpha^{\vee}(-m,-1)]\alpha^{\vee}(m,-1)\\
&+\alpha^{\vee}(-m,-1)[\alpha^{\vee}(s,1),\alpha^{\vee}(m,-1)]).w_{\lambda} \\
 =& \sum\limits_m a_m\delta_{s,m}(\alpha^{\vee}|\alpha^{\vee})\alpha^{\vee}(m,-1)sk_1.w_{\lambda}\\
=& 2a_ssk_1\alpha^{\vee}(1,-m).w_{\lambda}.
\end{aligned}\]
This contradiction proves our claim. \end{proof}

\section{\bf Highest weight modules of $\widehat{\widehat{\mathfrak{sl}}}_{2}$}
Futorny \cite{F} studied the imaginary Verma modules (IVM) for
affine Lie algebras and proved that an IVM is irreducible if and
only if $\lambda(c)\neq 0$. In this subsection, we prove
an irreducibility criterion for Verma modules
$M(\lambda)$ when $c_1\neq 0$ and $c_2=0$.
\begin{lemma}\label{4}
Let $0\neq v\in M(\lambda)$ and $M=\bigoplus\limits_{\eta\in
\mathcal{Q}_{1+}} M(\lambda)_{\lambda-\eta}$. Then $U(\mathfrak
T)v\cap M\neq 0 $.
\end{lemma}

\begin{proof}
Suppose $v\in M(\lambda)_{\lambda-\mu}$, where $\mu\in \mathcal
{Q}_+$. Let $\mu=\sum\limits_{i=0}^1 n_i\alpha_i+k\delta_2$ for
$n_i\in\mathbb{Z}$ and $k\in\mathbb{Z}_+$. Define the height of
$\mu$ by $ht(\mu)=k$. If $k=0$, then
$\lambda-\mu=\lambda-\sum\limits_{i=0}^1 n_{i}\alpha_i$ for some
nonnegative integers $n_i$, so the result holds. Suppose $k>0$,
since $M(\lambda)$ is a free $U(\mathfrak T_-)-$module, there exists
a homogenous element $u\in U(\mathfrak T_-)$ such that $v=u
v_{\lambda}$. By the PBW theorem
\[u=\sum_{p}X_{\phi_{1p}-n_{1p}\delta_2}^{l_{1p}}X_{\phi_{2p}-n_{2p}\delta_2}^{l_{2p}}\cdots
X_{\phi_{s(p)p}-n_{s(p)p}\delta_2}^{l_{s(p)p}}u_{p},\] where
$u_{p}\in U(\widehat{\mathfrak{g}_1}-)$,
$X_{\phi_{ip}-n_{ip}\delta_2}\in \mathfrak T_-$, $\phi_{ip}\in
\Delta_{\widehat{\mathfrak{g}_1}}$, $l_{ip}, n_{ip}\in\mathbb{N}$,
and $k=\sum_{i}n_{ip}l_{ip}~(i=1,2,\ldots,s(p))$ for all $p$. If
$i\neq j$, $\phi_{ip}-n_{ip}\delta_2\neq\phi_{jp}-n_{jp}\delta_2$
for all $p$. We will also assume $n_{1p}\geq n_{2p}\geq \cdots \geq
n_{s(p)p}$ for all $p$. Next we consider the set
$\Omega\subset\{n_{ip}\delta_2\}$ consisting of $\varphi_{ip}$ such
that $ht(-\varphi_{ip})=\min_p \{n_{s(p)p}\}$. In $\Omega$, we
consider the subset $\Omega^{'}$ consisting of $\varphi_{s(p)p}$
such that $l_{s(j)j}\geq l_{s(p)p}$ for all $\varphi_{ij}\in\Omega$.
We then take a subset $\Omega_0$ in $\Omega^{'}$ consisting of all
$\varphi_{s(p)p}$ such that $\phi_{ij}\geq\phi_{s(p)p}$. Without
loss of generality, we can assume $\phi_{s(1)1}\in\Omega_0$. Since
$\phi_{s(1)1}-n_{s(1)1}\delta_2=\phi_{ip}-n_{ip}\delta_2$ and
$l_{s(1)1}=l_{ip}$, we have $i=s(p)$. Choose sufficiently large
$\gamma\in\Delta_{\widehat{\mathfrak{g}_1}+}$ such that
$\gamma>\phi_{ij}$ for all $i, j$ and
$ht_{\widehat{\mathfrak{g}_1}+}(\gamma-\phi_{s(1)1})<ht_{\widehat{\mathfrak{g}_1}+}(-u_p)$
for all $p$. Let $0\neq z\in \mathfrak
T_{{\color{blue}-}\gamma+n_{s(1)1}\delta_2}$. We have
$zu_pv_{\lambda}=0$ for all $p$ since the weight of
$zu_pv_{\lambda}$ is larger than $\lambda$. The choice of
$n_{s(1)1}$ and $\gamma$ ensures that $zv\neq0$. Since
$ht(\mu+\gamma{\color{blue}-}n_{s(1)1}\delta_2)<ht(\mu)$, by
induction hypothesis we get $U(\mathfrak T)(zv)\cap M\neq 0$. Then
$U(\mathfrak T)v\cap M\neq 0$ because $U(\mathfrak T)(zv)\subset
U(\mathfrak T)v$.
\end{proof}

Let $M(\lambda)^+=\{v\in M(\lambda)|\mathfrak T_+.v=0\}$. Clearly it
is $\mathfrak{h}$-invariant. For an arbitrary nonzero element $v\in
M(\lambda)^+$, we see that $U(\mathfrak T_-).v$ is a submodule of
$M(\lambda)$. As for the form of elements in $M(\lambda)^+$, we have
the following result.
\begin{corollary}
$M (\lambda)^+\subset M$.
\end{corollary}
\begin{proof}
Suppose there exists a nonzero $v\in M(\lambda)^+$, and the weight
of $v$ is not of the form $\lambda-\beta~(\beta\in\mathcal
{Q}_{1+})$. Since $U(\mathfrak T)v= U(\mathfrak T_-)v$, the weight
of every element in $U(\mathfrak T)v$ can not be of the form
$\lambda-\gamma$ for any $\gamma\in\mathcal {Q}_{1+}$. Hence
$U(\mathfrak T)v\cap M=0$, which contradicts with Lemma \ref{4}.
\end{proof}

The subspace $M$ can be viewed as a Verma
$\widehat{\mathfrak{g}_1}-$module. Kac and Kazhdan \cite{KK} gave a
necessary and sufficient condition for the reducibility of the Verma
modules for affine Lie algebras.

\begin{theorem}\cite{KK}\label{5}
The Verma module $V(\lambda)$ of $\widehat{\mathfrak{g}_1}$ is
reducible if and only if for some positive root $\beta$ of the
algebra $\widehat{\mathfrak{g}_{1}}$ and some positive integer $l$,
one has $(\lambda+\rho)(\beta^{\vee})=l$, where
$\rho(\alpha_i^{\vee})=1~(i=0,1)$. Then $V(\lambda-l\beta)$ are
submodules of $V(\lambda)$.
\end{theorem}

\begin{theorem}\label{8}
The module $M(\lambda)$ is reducible if and only if for some
positive root $\beta$ of the algebra $\widehat{\mathfrak{g}_1}$ and
some positive integer $l$, one has $(\lambda+\rho)(\beta^{\vee})=l$,
where $\rho(\alpha_i^{\vee})=1~(i=0,1)$.
\end{theorem}

\begin{proof} Suppose that $M(\lambda)$ is reducible. Assume that on the contrary
that there does not exist a root $\beta$ of the algebra
$\widehat{\mathfrak{g}_1}$ and a positive integer $l$ such that
$(\lambda+\rho)(\beta^{\vee})=l$. By Theorem \ref{5}, $M$
is irreducible as a $\widehat{\mathfrak{g}_1}$-module. By Lemma
\ref{4}, for arbitrary $0\neq v\in M(\lambda)$, we have $U(\mathfrak
T)v\cap M\neq0$. Since $U(\mathfrak T)v\cap M$ is a
$\widehat{\mathfrak{g}_1}$-submodule of $M$, $U(\mathfrak T)v\cap M=M$ by the
$\widehat{\mathfrak{g}_1}$-irreducibility of $M$. Then $M\subset U(\mathfrak T)v$ and
$v_{\lambda}\in U(\mathfrak T)v$. Subsequently $U(\mathfrak T)v=
M(\lambda)$. Therefore $M(\lambda)$ is irreducible, which is a contradiction
to our assumption and
we have proved the necessary direction. 

On the other hand, if for some positive root $\beta$
of the algebra $\widehat{\mathfrak{g}_1}$,
$(\lambda+\rho)(\beta^{\vee})=l$ holds for a positive integer
$l$, then $M$ is reducible as a Verma module for
$\widehat{\mathfrak{g}_1}$ by Theorem \ref{5}, i.e., there exists
some singular vector $v\notin \mathbb Cv_{\lambda}$ of $M$.
Meanwhile, it is also a singular vector of $M(\lambda)$ since the
weight of $z(m,n).v~(z\in \mathfrak {sl}_2, m\in\mathbb{Z}, n\in
\mathbb{N})$ is higher than $\lambda$. Thus $M(\lambda)$ is
reducible.
\end{proof}

\begin{corollary}
Let $\lambda\in\mathfrak{h}^{*}$ such that
$\lambda(\alpha_i^{\vee})~(i=0,1)$ are nonnegative. Then
$W(\lambda)\cong L(\lambda)$.
\end{corollary}
\begin{proof}
Since $\lambda(\alpha_i^{\vee})~(i=0,1)$ are nonnegative,
 $\sum\limits_{i=0}^1U(\mathfrak
T_-).y_i^{n_i+1}v_{\lambda}$ is a maximal submodule of $M$ as
$\widehat{\mathfrak{g}_1}-$module by \cite{VG}. According to Theorem
\ref{5} and Theorem \ref{8}, we know that the reducibility of the
$U(\mathfrak{T})-$module $M(\lambda)$ is equivalent to that of $M$
as $\widehat{\mathfrak{g}_1}-$module. Hence, $W(\lambda)$ is
irreducible.
\end{proof}

\begin{corollary}
Let
$\Delta^{+}(\lambda)=\{(\beta,l)|(\lambda+\rho)(\beta^{\vee})=l, \beta\in\widehat{\mathfrak{g}_1}, l\in\mathbb N\}$.
Then
$J(\lambda)=\sum_{(\beta,l)\in\Delta^{+}(\lambda)}
M(\lambda-l\beta)$ is the maximal submodule of $M(\lambda)$. If
$V$ is the highest weight module of weight $\lambda$. Then $V\cong
M(\lambda)/N(\lambda)$, where $N(\lambda)\subset J(\lambda)$.
\end{corollary}

\vskip30pt \centerline{\bf ACKNOWLEDGMENTS}

This work is partially supported by 
Simons Foundation (No. 198129), NSFC grants (11271138, 11531004) and China Scholarship Council.

\end{document}